\documentclass[12pt, english,english]{article}
\usepackage{amsmath,amssymb,amsthm} 

\usepackage{mathrsfs,bm,enumerate}
\usepackage{graphicx,color,tikz}
\usepackage{caption,subcaption}
\usepackage{babel}
\usepackage[T1]{fontenc} 
\usepackage[latin9]{inputenc}

\newtheorem{lemma}{\textbf{Lemma}}
\newtheorem{corollary}{\textbf{Corollary}}
\newtheorem{proposition}{\textbf{Proposition}}
\newtheorem{theorem}{\textbf{Theorem}}
\newtheorem{conjecture}{\textbf{Conjecture}}

\theoremstyle{definition}
\newtheorem{definition}{\textbf{Definition}}

\providecommand{\abs}[1]{\left\lvert#1\right\rvert}
\providecommand{\norm}[1]{\left\rVert#1\right\rVert}
\providecommand{\set}[1]{ \left\{ #1  \right\}  }
\providecommand{\setb}[2]{ \left\{ #1 \ \middle| \  #2 \right\}  }
\providecommand{\innprod}[2]{\left\langle #1, #2 \right\rangle}
\providecommand{\parenth}[1]{\left( #1 \right) }

\providecommand{\expect}[1]{\mathbb{E}\left[ #1 \right] }
  
\def\CF{{\widehat{\mathscr{P}}}}
\def\D{{\mathcal{D}}}
\def\S{{\mathcal{S}}}
\def\R{{\mathcal{R}}}
 
\def\C{ \mathbb{C}}
\def\Z{ \mathbb{Z}}

\def\N{ \mathbb{N}}
\def\R{ \mathbb{R}}
\def\rbf{ \mathbf{r}}
\def\drm{\mathrm{d}}
\def\Rstar{ \mathbb{R} \backslash \{0\}}
\def\ui{\mathrm{i}}
\def\ue{\mathrm{e}}

\def\T{\mathbb{T}}

\date{}

\title{On the Besov Regularity of Periodic L\'evy Noises
  \thanks{The research leading to these results has received funding from the European Research Council under the European Union's Seventh Framework Programme (FP7/2007-2013) / ERC grant agreement $\text{n}^\circ$ 267439.}}

\author{
Julien Fageot  \thanks{Biomedical Imaging Group, \'Ecole polytechnique f\'ed\'erale de Lausanne (EPFL),
Station 17, CH-1015, Lausanne, Switzerland ({\tt julien.fageot@epfl.ch, michael.unser@epfl.ch, john.ward@epfl.ch}). }
\and
Michael Unser \footnotemark[2]
\and
John Paul Ward \footnotemark[2]
 }

\begin{document}

\maketitle

\begin{abstract}
In this paper, we study the Besov regularity of L\'evy white noises on the $d$-dimensional torus. Due to their rough sample paths, the white noises that we consider are defined as generalized stochastic fields. We, initially, obtain regularity results for general L\'evy white noises.  Then, we focus on two subclasses of noises: compound Poisson and symmetric-$\alpha$-stable (including Gaussian), for which we make more precise statements. Before measuring regularity, we show that the question is well-posed; we prove that Besov spaces are in the cylindrical $\sigma$-field of the space of generalized functions. These results pave the way to the characterization of the $n$-term wavelet approximation properties of stochastic processes.
\end{abstract}

%%%%%%%%%%%%%%%%%%%%%%%%%%%%%%%%%%%%%%%%%%%%%%%%%%%%%%%%%%%%%%%%
\section{Introduction} \label{sec:intro}
%%%%%%%%%%%%%%%%%%%%%%%%%%%%%%%%%%%%%%%%%%%%%%%%%%%%%%%%%%%%%%%%

The classical theory of stochastic processes deals with the definition and study of pointwise processes $(X_t)_{t\in T}$, where $T$ is a continuous-domain index set. 
For $T = \R$, L\'evy processes are important examples that are stochastically continuous with stationary and independent increments.
Conventionally, these processes are constructed by providing the probability law of their finite-dimensional marginals $(X_{t_1},\cdots,X_{t_N})$. 
Actually, they are fully characterized by the law of $X_1$, which is infinitely divisible \cite{Sato94}.
In this paper, we are using a more abstract approach. 
For us, a process is a probability measure on the space of generalized functions, and 
as such, it is called a generalized stochastic process. 
We are following the seminal work of Gelfand and Vilenkin (Chapter III of \cite{GelVil4}), and this approach
has at least two advantages:
1) it allows us to give a meaning to processes that have no classical definition, such as white noises or their derivatives 
and 2)  the solutions of stochastic partial differential equations are defined in a very general way.
Along with these advantages, we face the difficulty of not knowing \emph{a priori} anything about the regularity of generalized stochastic processes. The question of regularity is thus a crucial one.
For instance, the regularity of a process is directly related to its $n$-term wavelet approximation \cite{cohen2000,Devore1998Nterm,temlyakov1998} or tree approximation \cite{baraniuk2002,cohen2001}. 

In this paper, we propose to determine the regularity of white noises defined on the $d$-dimensional torus $\T^d$.
We obtain regularity results for general L\'evy noises, and we make more precise statements for the classes of compound Poisson and symmetric-$\alpha$-stable white noises.
Our reason for working with white noises and not, for instance, pointwise L\'evy processes, is that there is not a unique way to define L\'evy processes indexed by $\mathbf{r}\in \R^d$ with $d \geq 2$; in fact, there are several definitions of the so-called L\'evy fields  that generalize the one-dimensional setting.
However, for each such definition, a $d$-dimensional L\'evy field is obtained by integrating its corresponding L\'evy white noise.
Moreover, we prefer to study L\'evy white noises on $\T^d$, over those on $\R^d$, in order to focus solely on the question of regularity.
Indeed, the study of functions on $\R^d$ mixes the questions of regularity and decay at infinity.
For stochastic processes, understanding the decay properties is already challenging 
(consider the log iterate law for Brownian motion \cite{Karatzas2012brownian}).

We measure the regularity of periodic stochastic processes in terms of Besov spaces, which generalize certain families of Sobolev spaces. 
In general, they provide a finer measure of smoothness than Sobolev and H\"{o}lder spaces, and they often appear in Banach space interpolation, when interpolating between classical smoothness spaces. Our primary interest in Besov spaces is their connection to wavelet analysis.
Wavelet bases are composed of scaled versions of a mother wavelet, and for any given Besov space, we can define a wavelet basis. This  is very convenient for studying stable processes, as we shall see.

To the best of our knowledge, the specific question of the Besov regularity of general periodic L\'evy noises has not been addressed in the literature.  
However, there are numerous works dealing with the regularity of L\'evy-type processes, and these are closely related to ours.
On the regularity of Gaussian processes, we mention  \cite{bel1998linear,benyi2011modulation,ciesielski1993quelques,roynette1993mouvement}. 
The Sobolev regularity of non-periodic Gaussian white noises was first addressed by Kusuoka \cite{kusuoka1982support}.
More recently, Veraar obtained important results on the Besov regularity of periodic Gaussian white noises \cite{veraar2010regularity};
our work extends some of these results to the case of non-Gaussian noises. 
Stable classical processes were studied in \cite{ciesielski1993quelques,huang2007fractional}. 
More generally, Herren and Schilling examined the local Besov regularity of L\'evy processes in one dimension  \cite{herren1997levy,Jacob2001levy,schilling1997Feller,Schilling2000function}.
The works of Schilling include results on the global regularity of L\'evy processes in terms of weighted Besov spaces and extend to the case of Feller processes.
For a review on the sample path properties of one dimensional L\'evy processes, see \cite[Chapter 5]{Bottcher2013levy}; see also \cite{Barndorff2001levy}.
As we said, there are different definitions of classical L\'evy fields that generalize the one-dimensional L\'evy processes
(cf. \cite[Section 4]{fageot2014} for a white noise approach).
For discussions on this question, we refer the reader to \cite{dalang1992, durand2012multifractal}. 

Our work is motivated by the recent development of a probabilistic model for natural images, the \emph{innovation model} \cite{Unser2014sparse}. Images are described as solutions of stochastic partial differential equations driven by a L\'evy white noise. 
When the white noise is non-Gaussian, the processes of the innovation model are called \emph{sparse processes}, to which the results of this paper can be applied. Indeed, the regularity of L\'evy noises is the first question to address in order to understand the regularity of sparse processes. Note that the theory in \cite{Unser2014sparse} is not periodic; however, the present work is an important first step 
toward the understanding of sparse processes in general.

This paper is organized as follows. In Section \ref{sec:maths}, we introduce the main notions for our study: 
the definition of periodic generalized stochastic processes and Besov spaces.
Then, we show that the question of the Besov regularity of a generalized process is well-posed. 
Technically, we show that Besov spaces are part of the cylindrical $\sigma$-field of $\S'(\T^d)$ (Section \ref{measurable}). 
Taking advantage of these new results, we determine the Besov regularity of periodic white noises, 
with a special emphasis on stable noises  (Section \ref{noise_regularity}).
The last section contains a discussion of our results and some possible extensions.

%%%%%%%%%%%%%%%%%%%%%%%%%%%%%%%%%%%%%%%%%%%%%%%%%%%%%%%%%%%%%%%%
\section{Mathematical Foundations} \label{sec:maths}
%%%%%%%%%%%%%%%%%%%%%%%%%%%%%%%%%%%%%%%%%%%%%%%%%%%%%%%%%%%%%%%%

\subsection{Periodic L\'evy White Noises}
	 
In this paper, we study the regularity of periodic L\'evy white noises. 
It is well-known that white noises cannot be defined as pointwise stochastic processes. 
Nevertheless, they can be defined as random generalized functions \cite{GelVil4}.
Note that Gelfand and Vilenkin only define non-periodic white noises, as generalized processes, \emph{i.e.} random elements of $\D'(\R^d)$, the space of generalized functions.
Thus, we extend \cite{GelVil4} to periodic generalized processes and periodic white noises so that 
 a periodic white noise is defined as a random element of $\S'(\T^d)$, the space of generalized functions on the $d$-dimensional torus $\T^d= \R^d / \Z^d$.

For a proper definition of periodic generalized processes, we introduce a measurable structure on the space of periodic generalized functions. The space of smooth periodic functions is denoted by $\S(\T^d)$, with topological dual $\S'(\T^d)$. For $u\in \S'(\T^d)$ and $\varphi \in \S(\T^d)$, we denote their duality products as $\innprod{u}{\varphi}$, assumed to be linear in both components. We use the same notation, in general, for pairs of elements from topological dual spaces. For  $(u,\varphi) \in E'\times E$ with $\S(\T^d) \subset E$ and $E'\subset \S'(\T^d)$, we write $\innprod{u}{\varphi}$, with the test function $\varphi$ in the second position. 

We denote probability spaces as $(\Omega, \mathcal{F}, \mathscr{P})$. 
A cylindrical set of $\S'(\T^d)$ is a subset of the form
\begin{equation}
	\setb{ u \in \S'(\T^d) }
	{ \parenth{ \innprod{u}{\varphi_1} , \cdots, \innprod{u}{\varphi_N} } \in B },
\end{equation}
with $N\in \N$, $(\varphi_1,\cdots,\varphi_N) \in (\S(\T^d))^N$ and $B$ a Borel subset of $\R^N$. The cylindrical $\sigma$-field of $\S'(\T^d)$ is the $\sigma$-field $\mathcal{B}_c(\S'(\T^d))$ of $\S'(\T^d)$ generated by the cylindrical sets. 
\begin{definition}
A \emph{generalized periodic process on $\Omega$} is a measurable function 
\begin{equation}
	s  :  \parenth{\Omega, \mathcal{F}} 
	\rightarrow 
	\parenth{\S'(\T^d), \mathcal{B}_c \parenth{\S'(\T^d)}}.
\end{equation}
The \emph{probability law of $s$} is the probability measure $\mathscr{P}_s$ on the measurable space $ \parenth{\S'(\T^d), \mathcal{B}_c \parenth{\S'(\T^d)}}$:
\begin{equation}
	\forall B 
	\in  
	\mathcal{B}_c(\S'(\T^d)),
	\quad
	\mathscr{P}_s(B) 
	:= 
	\mathscr{P} \parenth{ \setb{ \omega \in \Omega}{  s(\omega) \in B } }.
\end{equation}
The \emph{characteristic functional} of $s$ is the Fourier transform of its probability law, given for $\varphi \in \S(\T^d)$ by
\begin{equation}
	\CF_s(\varphi)
	:= 
	\int_{\S'(\T^d)} 
	\ue^{\ui \langle u,\varphi\rangle} 
	\drm \mathscr{P}_s(u).
\end{equation}
\end{definition}

The characteristic functional characterizes the law of $s$. Because the space $\S(\R^d)$ is nuclear, we have the following result, known as the Minlos-Bochner theorem. 

\begin{theorem} [Theorem 3, Section III-2.6, \cite{GelVil4}] 
A functional $\CF : \S(\T^d)\rightarrow \C$ is the characteristic functional of some generalized process $s$ if and only if $\CF$ is continuous, positive-definite, and $\CF(0)=1$. 
\end{theorem}

A function $f:\R \rightarrow \C$ is said to be a L\'evy exponent if it is the log-characteristic function of an infinitely divisible law. To be precise, L\'evy exponents $f$ are continuous functions that are conditionally positive definite of order one with $f(0)=0$ \cite[Section 4.2]{Unser2014sparse}. 
The L\'evy-Khintchine theorem (cf. Theorem 8.1 of \cite{Sato94}) 
ensures that there exists $\mu \in \R$, $\sigma^2 \geq 0$ and a L\'evy
measure $V$ ---\emph{i.e.}, a Radon measure on $\Rstar$ with $\int_{\Rstar} \min(1,t^2) V(\drm t) < \infty$--- such
that
\begin{equation} \label{eq:LK}
	f(\xi) 
	= 
	\ui \mu \xi 
	- \frac{\sigma^2 \xi^2}{2} 
	+ \int_{\Rstar} 
	\parenth{\ue^{\ui \xi t} - 1 - \ui \xi t 1_{|t|\leq1} } 
	V(\drm t).
\end{equation}

A L\'evy white noise $w$ on $\D'(\R^d)$ is defined from its characteristic functional by the relation,  
for all $\varphi \in \D(\R^d)$, 
\begin{equation}
	\CF_w(\varphi) 
	= 
	\exp \parenth{ 
	\int_{\R^d} 
	f(\varphi(\rbf)) 
	\drm\rbf }, 
\end{equation}
with $f$ a L\'evy exponent (see \cite[Theorem 6, Section III-4.4]{GelVil4}).
L\'evy white noises are stationary and independent at every point (which means that $\innprod{w}{\varphi_1}$ and $\innprod{w}{\varphi_2}$ are independent if the supports of $\varphi_1,\varphi_2\in \D(\R^d)$ are disjoint).
Following the same idea, we now define periodic L\'evy white noises.

\begin{definition}
A generalized periodic process $w$ is a \emph{periodic L\'evy white noise} if its characteristic functional has the form 
\begin{equation} \label{eq:CF}
	\CF_w(\varphi) 
	= 
	\exp 
	\parenth{ 
	\int_{\T^d} 
	f(\varphi(\rbf)) 
	\drm\rbf }, \quad 
	\forall \varphi \in \S(\T^d),
\end{equation}
with $f$ a L\'evy exponent. 
\end{definition}
The proof that the previous functional $\CF_w$ is a well-defined characteristic functional on $\S(\T^d)$ is deduced from the case of white noises over $\D'(\R^d)$ (see \cite[Theorem 5, Section III-4.3]{GelVil4}). Essentially, this reduces to the fact that the periodization of test functions from $\D(\R^d)$ are test functions in $\S(\T^d)$.

\subsection{Smoothness Spaces}

While our goal is to characterize the smoothness of white noise processes in terms of Besov regularity, we make use of Sobolev spaces as an intermediate step.  The Sobolev space $H_2^{\tau}(\T^d)$ of order $\tau \in \R$ is defined as follows.

\begin{definition}
A periodic generalized function $f$ is in $H_2^{\tau}(\mathbb{T}^d)$ if its norm, defined by 
\begin{equation}
	\norm{f}_{H_2^{\tau}}
	:=
	\sqrt{
	\sum_{\bm{n}\in \Z^d} 
	(1 + \abs{\bm{n}})^{2\tau} 
	\abs{ 
	\innprod{f}{ 
	\ue^{- 2\pi \ui \langle \bm{n},  \cdot \rangle} 
	} }^2},  
\end{equation}
is finite.
\end{definition}

Besov spaces can be used to accurately characterize the smoothness of periodic generalized functions.  One of the advantages of using these spaces is that they are well suited to wavelet analysis.  We can define wavelet bases for Besov spaces, and the wavelet analysis provides a means to define an isomorphism between Besov spaces on the domains $\mathbb{T}^d$ and $\mathbb{N}$.  

For clarity in the definition, we index the Besov sequences in a manner analogous to the wavelet index. The index $j$ denotes the scaling index of the wavelet, and it ranges from $0$ to $\infty$.  We adjust the scaling so that the scale $0$ wavelets are supported within the torus, and $L\in\mathbb{N}$ is used for this purpose.  The index $G$ is used to refer to the gender of the wavelet.  Since the coarsest scale includes the scaling function, the set of indices $G^0$ at scale $0$ has $2^d$ elements. For $j>0$, $G^j$ has $2^d-1$ elements. The translation sets are denoted as  
\begin{equation}
	\mathbb{P}_j^d 
	=
	\setb{ 
	m\in\mathbb{Z}^d}{
	0\leq m_k<2^{j+L},
	1 \leq k  \leq d 
	}, 
	\quad 
	j\in\mathbb{N}_0.
\end{equation}
These are the $2^{(j+L)d}$ lattice points in $2^{(j+L)}\mathbb{T}^d$. 
Finally, we denote the smoothness of a wavelet by the index $u\in\mathbb{N}$; \emph{i.e.}, $u$ denotes the number of continuous derivatives of a given wavelet basis.
For a given smoothness $u$, we can construct compactly supported wavelets that are orthonormal bases of $L_2(\mathbb{T}^d)$, cf. \cite[Section 1.3.2]{triebel08}. We consider only wavelets that are real valued. 

\begin{proposition}[Proposition 1.34, \cite{triebel08}]
If $u\in\mathbb{N}$, then there is a wavelet system 
\begin{equation}
	\setb{\Psi_{G,m}^{j,\mathrm{per}}}{ j\in\mathbb{N}_0,\ G\in G^j,\ m\in\mathbb{P}_j^d }
\end{equation}
with regularity $u$ that is an orthonormal basis of $L_2(\mathbb{T}^d)$. \end{proposition}

We exclusively use the wavelet characteriztion of Besov spaces. For further details on alternative definitions, see \cite[Section 3.5]{schmeisser87}. We begin by specifying the Besov sequence spaces.

\begin{definition}[Definition 1.32, \cite{triebel08}]
Let $\tau \in \mathbb{R}$, $0<p,q\leq \infty$. Then the Besov sequence space $b_{p,q}^{\tau,\mathrm{per}}$ is the collection of all sequences 
\begin{equation}
	\lambda 
	=
	\setb{ \lambda_{m}^{j,G} \in\mathbb{C}}{ 
	j\in\mathbb{N}_0, G\in G^j,m\in\mathbb{P}_j^d }
\end{equation}
such that 
\begin{equation}
	\norm{\lambda }_{b_{p,q}^{\tau,\mathrm{per}}} 
	:=
	\parenth{ 
	\sum_{j=0}^{\infty} 2^{j(\tau-d/p)q} 
	\sum_{G\in G^j} 
	\parenth{ 
	\sum_{m\in \mathbb{P}_j^d }\abs{\lambda_m^{j,G}}^p 
	}^{q/p}  
	}^{1/q}
	<
	\infty.
\end{equation}
When $p,q=\infty$, suitable modifications must be made to the norm.
\end{definition}

The following is a characterization of periodic Besov spaces. 

\begin{theorem}[Theorem 1.37, {\cite{triebel08}}] \label{theo_triebel}
Let $\{\Psi_{G,m}^{j,\mathrm{per}}\}$ be an orthonormal basis of $L_2(\mathbb{T}^d)$  with regularity $u\in\mathbb{N}$. Let $0<p,q\leq \infty$ and $\tau_0\in \mathbb{R}$ such that  $u>\max \left(\tau_0,d\left(1/p-1\right)_+-\tau_0\right)$. Let $f\in \S^\prime (\mathbb{T}^d)$. Then $f\in B_{pq}^{\tau_0}(\mathbb{T}^d)$ if, and only if, it can be represented as 
\begin{equation}
	f 
	= 
	\sum_{j,G,m} 
	\lambda_m^{j,G}2^{-(j+L)d/2}
	\Psi_{G,m}^{j,\mathrm{per}}, \quad 
	\lambda
	\in b_{p,q}^{\tau_0,\mathrm{per}},
\end{equation}
which converges unconditionally in $\S^\prime (\mathbb{T}^d)$ and in any space $B_{p,q}^{\tau}(\mathbb{T}^d)$ with $\tau<\tau_0$. Furthermore, this representation is unique,
\begin{equation}
	\lambda_m^{j,G}
	=
	\innprod{f}{
	2^{(j+L)d/2}\Psi_{G,m}^{j,\mathrm{per}} },
\end{equation}
and 
\begin{equation*}
	I: 
	f
	\mapsto 
	\set{   \innprod{f}{ 2^{(j+L)d/2}\Psi_{G,m}^{j,\mathrm{per}} }  }
\end{equation*}
is an isomorphic map of $B_{p,q}^{\tau_0}(\mathbb{T}^d)$ onto $b_{p,q}^{\tau_0,\mathrm{per}}$.
\end{theorem}

Besov spaces are Banach spaces when $p,q \geq 1$ and quasi-Banach spaces otherwise.
Connections between Besov spaces with different parameters are provided below. For the benefit of the reader, we recall the necessary results from \cite{schmeisser87}.
We follow the notation of this reference for (quasi-)Banach space embeddings.

\begin{definition}
A (quasi-)Banach space $A_1$ is continuously embedded in another (quasi-)Banach space $A_2$, denoted as $A_1\subset A_2$, if there exists a constant $C>0$ such that 
\begin{equation}
	\norm{f}_{A_2}
	\subset 
	C\norm{f}_{A_1}
\end{equation}
for every $f\in A_1$.
\end{definition}

\begin{proposition}[p.164-170, \cite{schmeisser87}]\label{prop:embed}
If $\tau \in \R$ and $\epsilon>0$, then 
\begin{alignat}{3} \label{eq:embeddings}
	B_{2,2}^{\tau}(\T^d) 
	&= 
	H_{2}^{\tau}(\T^d),  \nonumber  \\
	B_{p_0,q_0}^{\tau+d(1/{p_0}-1/{p_1})_{+}}(\T^d) 
	&\subset 
	B_{p_1,q_1}^{\tau-\epsilon}(\T^d), 
	&
	\quad
	0 
	&<
	p_0 , p_1 \leq \infty,
	&
	\quad
	0
	&< q_0,q_1\leq \infty, \nonumber \\
	B_{p,q}^{\tau+d(1/p-1/2)_{+}}(\T^d)
	&\subset
	H_{2}^{\tau-\epsilon}(\T^d),
	&
	0
	&<
	p \leq \infty,
	&
	0
	&< q \leq \infty, \nonumber \\
	H_{2}^{\tau+d(1/2-1/p)_{+}}(\T^d)
	&\subset
	B_{p,q}^{\tau-\epsilon}(\T^d),
	&
	0
	&<
	p \leq \infty,
	&
	0
	&< q \leq \infty.
\end{alignat} 
\end{proposition}

In the sequel, we use a simplified version of Theorem \ref{theo_triebel}, making the connection with Sobolev spaces. Before stating this result, we recall the duals of Besov spaces.  For $0<p<\infty$, the conjugate $p'$ is the solution of $1/p + 1/p' = 1$ if $p\geq 1$, and $p' = \infty$ if $p<1$. 
\begin{theorem}[p.171, \cite{schmeisser87}] \label{theo:dual}
Let $\tau\in\mathbb{R}$, $0<p,q<\infty$. Then, we have
\begin{itemize}
\item $\left(B_{p,q}^\tau(\mathbb{T}^d) \right)' = B_{p',q'}^{-\tau}(\T^d)$ if $1\leq p <\infty$,
\item $\left( B_{p,q}^\tau(\mathbb{T}^d)\right)' = B_{\infty,q'}^{-\tau+d(1/p-1)}(\mathbb{T}^d)$ if $0<p<1$.
\end{itemize}
\end{theorem}

\begin{figure}[h!] 
\centering
\begin{tikzpicture}[x=4cm,y=3cm,scale=0.75]

\draw[very thick, ->] (-1,0)--(0.5,0) node[circle,right] {$\frac{1}{p}$} ;
\draw[very thick, <->] (-1,-1)--(-1,0.5) node[circle,above] {$\tau$} ;

\draw[]
(-0.5,-0.5) 
node[black, circle, minimum width=4pt, fill,fill opacity=1,  inner sep=0pt]{}
node[black, below] {$(1/p_0,\tau_0) $}
; 

\draw[ thick,color=black]
(0,0.05) -- (0,-0.05)  node[black, below] { $1$};
\end{tikzpicture}
\hspace*{1cm}
\begin{tikzpicture}[x=4cm,y=3cm,scale=0.75]

\draw[very thick, ->] (-1,0)--(0.5,0) node[circle,right] {$\frac{1}{p}$} ;
\draw[very thick, <->] (-1,-1)--(-1,0.5) node[circle,above] {$\tau$} ;

\fill[red, opacity= 0.4] (-1,-0.5) -- (0,0.25) -- (0.5,0.25) -- (0.5,-1) -- (-1,-1) -- cycle; 
\draw[red, very thick, dashed,->](0,0.25) --(0.5,0.25);
\draw[red, very thick, dashed,- ](-1,-0.5) -- (0,0.25);
\draw[red, very thick,->](-1,-0.5) -- (-1,-1);

\draw[]
(0,0.25) 
node[black, circle, minimum width=4pt, fill,fill opacity=1,  inner sep=0pt]{}
node[black, above] {$B_{p_0,q_0}^{\tau_0}(\T^d) $};

\end{tikzpicture}

\caption{ On the left: this diagram provides a graphic representation of Besov spaces, including the Sobolev spaces.  The vertical axis corresponds to the smoothness parameter $\tau$, while the horizontal axis corresponds to the inverse of the ``singularity'' parameter $p$.  A given point $(1/p_0,\tau_0)$ represents the spaces  $B_{p_0,q}^{\tau_0}(\T^d)$  for $0< q \leq \infty$. 
On the right: this diagram illustrates the second item of Proposition \ref{prop:embed}. The plotted Besov space $B_{p_0,q_0}^{\tau_0}(\T^d)$ at $(1/p_0,\tau_0)$ embeds in all of the Besov spaces $B_{p_1,q_1}^{\tau_1}(\T^d)$ below the Sobolev embedding line (with slope $d$) for $p_1 \geq p_0$, and it embeds in the spaces $B_{p_1,q_1}^{\tau_1}(\T^d)$ below the line of slope $0$  for $p_1<p_0$.} 
\label{fig_Besov}
\end{figure}
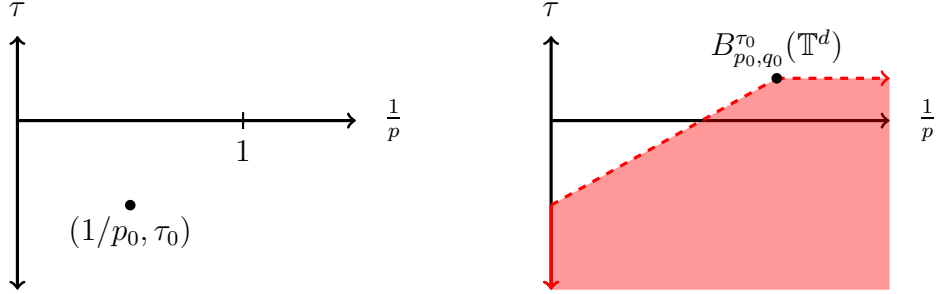

\begin{corollary}  \label{theo:isomorphism}
Let $\tau \in \R$, $0<p \leq \infty$, $\epsilon>0$, and $u\in \N$ such that $u>\max(\tau, d(1/p-1/2)_+-\tau+\epsilon)$. Let $\{\Psi_{G,m}^{j,\mathrm{per}}\}$ be an orthonormal basis of $L_2(\mathbb{T}^d)$  with regularity $u$. Then, $B_{p,q}^{\tau}(\T^d)$ is the subset of $H_2^{\tau-d(1/p-1/2)_+-\epsilon}(\T^d) $ corresponding to functions $f$ satisfying
\begin{equation*}
	\set{   \innprod{f}{2^{(j+L)d/2}\Psi_{G,m}^{j,\mathrm{per}} }  } \in b_{p,q}^{\tau,\mathrm{per}}.
\end{equation*}
\end{corollary}
\begin{proof}
The inequality $u>\max(\tau, d(1/p-1/2)_+-\tau+\epsilon)$ implies that $u$ is greater than the three quantities $\tau$, $d\left(1/p-1\right)_+-\tau$, and $d(1/p-1/2)_+-\tau+\epsilon$.
The first two conditions are required to apply Theorem \ref{theo_triebel}. The last condition guarantees that the wavelets are included in the dual of the Sobolev space $H_2^{\tau-d(1/p-1/2)_+-\epsilon}(\T^d)$.
\end{proof}

\begin{corollary}  \label{theo:isomorphism_2}
Let $\epsilon>0$, $f\in H_2^{-d/2-\epsilon}(\T^d)$, $\tau \in \R$, $0<p,q \leq \infty$, and $u\in \N$ such that 
$u>\max(d/2+\epsilon,\tau, d(1/p-1/2)_+-\tau+\epsilon)$. Let $\{\Psi_{G,m}^{j,\mathrm{per}}\}$ be an orthonormal basis of $L_2(\mathbb{T}^d)$  with regularity $u$. Then $f\in B_{p,q}^{\tau}(\T^d)$ iff
\begin{equation}
	\set{  \innprod{f}{2^{(j+L)d/2}\Psi_{G,m}^{j,\mathrm{per}} }  } 
	\in 
	b_{p,q}^{\tau,\mathrm{per}}.
\end{equation}
\end{corollary}
\begin{proof}
We have the inclusion $f\in H_2^{-d/2-\epsilon}(\mathbb{T}^d)\subset H_2^{\tau_1}(\mathbb{T}^d)$, 
where 
\begin{equation}
\tau_1 = \min\set{-d/2-\epsilon,\tau-d \parenth{\frac{1}{p}-\frac{1}{2}}_+-\epsilon}.
\end{equation} 
Therefore, Theorem \ref{theo_triebel} implies that $f$ has a wavelet expansion with coefficients in $b_{2,2,d}^{\tau_1,\mathrm{per}}$. We also know that $B_{p,q}^{\tau}(\mathbb{T}^d)\subset H_2^{\tau_1}(\mathbb{T}^d)$, so we apply Corollary \ref{theo:isomorphism} to complete the proof. 
\end{proof}
This corollary can be summarized as follows. We have a function $f$ in $H_2^{-d/2-\epsilon}(\mathbb{T}^d)$, and we want to know if it is in a given Besov space $B_{p,q}^{\tau}(\T^d)$. We therefore determine a Sobolev space $H_2^{\tau_1}(\mathbb{T}^d)$ that both $H_2^{-d/2-\epsilon}(\mathbb{T}^d)$ and $B_{p,q}^{\tau}(\T^d)$ are embedded in.  The coefficients of the wavelet basis for $H_2^{\tau_1}(\mathbb{T}^d)$ then characterize the $B_{p,q}^{\tau}(\T^d)$ smoothness.

\begin{proposition}[p.164, \cite{schmeisser87}] \label{proposition_dirac}
If $\tau<d(1/p-1)$ and $0 < p,q \leq \infty $, then the Dirac distribution $\delta$ is in the Besov space $B_{p,q}^{\tau}(\T^d)$.
\end{proposition}

%%%%%%%%%%%%%%%%%%%%%%%%%%%%%%%%%%%%%%%%%%%%%%%%%%%%%%%%%%%%%%%%
\section{Besov Spaces are Measurable} \label{measurable}
%%%%%%%%%%%%%%%%%%%%%%%%%%%%%%%%%%%%%%%%%%%%%%%%%%%%%%%%%%%%%%%%

The framework of generalized stochastic processes allows us to define very general processes, 
such as white noises or (weak) derivatives of white noises of any order. 
Consequently, the problem of existence of a solution of a stochastic partial differential equation is more easily solved than in the classical framework. 
However, we only know that our solution is a random element in the whole space $\S'(\T^d)$ of periodic generalized functions. 
In particular, we  know \emph{a priori} nothing about the regularity of a generalized stochastic process. 
We would like to say, for instance, that a process is continuous if its sample paths are almost surely (a.s.) in the space of continuous functions on the torus. 
More generally, we would like to understand the regularity in terms of Besov spaces. 
However, to ask the question ``Is my process in a given Besov space with probability one?'', 
we first need to show that this question is well-posed. 
This requires us to show that Besov spaces are measurable subspaces of the space of generalized functions. 

\begin{theorem}\label{theorem:measurable} (Besov spaces are measurable.)
For every $0<p,q \leq +\infty$ and every $\tau \in\R$, we have
\begin{equation} \label{eq:besov_measurable}
	B_{p,q}^\tau(\T^d) 
	\in 
	\mathcal{B}_c(\S'(\T^d)),
\end{equation}
where we remind the reader that $\mathcal{B}_c(\S'(\T^d))$ is the cylindrical $\sigma$-field of $\S'(\T^d) $.
\end{theorem}

We prove this result in two steps. First, we focus on the Sobolev spaces $H_2^{v}(\T^d)$ for $v\in \R$, and we show that they are measurable subspaces of $\S'(\T^d)$. Then, we reduce the question to the fact that every Besov space is a measurable subspace of some Sobolev space. 
This simplification step is useful because we have a wavelet characterization of Besov spaces as subspaces of Sobolev spaces. In order to apply the described approach, we use the following result. 
\begin{lemma}  \label{lemma:measurable}
Let $E$ be a topological vector space, $E'$ its topological dual and $\mathcal{B}_c(E')$ the cylindrical $\sigma$-field on $E'$, generated by the cylinders of the form 
\begin{equation}
	C_{\varphi_1, \cdots ,\varphi_N, B} 
	:= 
	\setb{ f \in E' }{ \parenth{ \innprod{f}{\varphi_1}, \cdots , \innprod{f}{\varphi_N} } \in B}.
\end{equation}
Here,  $N \in \N\backslash\{0\}$, $\varphi_1,\cdots,\varphi_N \in E$ and $B$ is a Borel subset of $\R^N$.
Then, for  every countable set $S$, every collection of finite sets $T_n$ ($n \in S$), every  $\varphi_n,\varphi_{n,m} \in E$ ($n\in S$ and $m\in T_n$), and every $\alpha, \beta>0$, we have
\begin{equation} \label{cylinder_1}
	\setb{f \in E' }{ \sum_{n\in S} \abs{ \innprod{f}{\varphi_n} }^\alpha <\infty} 
	\in 
	\mathcal{B}_c(E')
\end{equation}
and
\begin{equation} \label{cylinder_2}
	\setb{f \in E' }{ \sum_{n\in S} \left( \sum_{m\in T_n} |\langle f, \varphi_{n,m}\rangle|^\alpha \right)^\beta <\infty }
	\in 
	\mathcal{B}_c(E').	
\end{equation}
\end{lemma}

\begin{proof}
We prove \eqref{cylinder_2} and deduce \eqref{cylinder_1} by choosing $T_n$ to have cardinality $1$ with $\beta = 1$. 
We denote by $\R^{\N}$ the space of real sequences, with the product $\sigma$-field. 
The set of couples $(n,m)$ with $n\in S$ and $m\in T_n$ is denoted by $A$.
By definition of the cylindrical $\sigma$-field, for fixed $\boldsymbol{\varphi}= (\varphi_{n,m})_{(n,m)\in A}$, the projection 
\begin{equation}
	\pi_{\boldsymbol{\varphi}} (f) 
	= 
	\parenth{\innprod{f}{\varphi_{n,m}}}_{(n,m)\in A}
\end{equation}
is measurable from $E'$ to $\R^\N$. 
(Rigorously, this projection is from $E'$ to $\R^A$, but $A$ being countable, we admit this slight abuse of notation.)
Moreover, the function $F_{\alpha,\beta}$ from $\R^\N$ to $\R^+ \cup \{\infty\}$ that associates to a sequence $(a_{n,m})_{(n,m)\in A}$ the quantity $\sum_{n\in \N} \left( \sum_{m\in T_n} |a_{n,m}|^\alpha \right)^\beta$ is measurable. 
Finally,  since $\R^+$ is measurable in $\R^+ \cup \{\infty\}$,
\begin{equation}
	\setb{ f \in E' }{ \sum_{n\in \N} \parenth{ \sum_{m\in T_n} \abs{\innprod{f}{\varphi_{n,m}}}^\alpha }^\beta <\infty }
	= 
	\pi_{\boldsymbol{\varphi}}^{-1} \parenth{F_{\alpha,\beta}^{-1} (\R^+)}
\end{equation}
 is measurable in $E'$. 
\end{proof}

\noindent
\textit{Proof of Theorem \ref{theorem:measurable}.}

\noindent
\textbf{Step 1: $H_2^{v}(\T^d)$ is measurable.}

For a fixed $v\in \R$, we have
\begin{equation}
	H_2^{v}(\T^d) 
	= 
	\setb{ f \in \S'(\T^d)}{ 
	\sum_{\bm{n}\in \Z^d} 
	\parenth{1 + \abs{\bm{n}}}^{2v} 
	\abs{\innprod{ f}{\ue^{-2\pi \ui \langle \bm{n}, \cdot \rangle} }}^2 < \infty }.
\end{equation}
Thus, the conditions of Lemma \ref{lemma:measurable}
are satisfied,  with $E = \S(\T^d)$, $S = \Z^d$, $\alpha = 2$ and $\varphi_{\bm{n}} = \parenth{1+\abs{\bm{n}}}^{v} \ue^{-2\pi\ui  \langle \bm{n}, \cdot \rangle }$, and we obtain that $H_2^{v}(\T^d) \in \mathcal{B}_c(\S'(\T^d))$.

Moreover, we know that $\S(\T^d) = \bigcap_{k\in \mathbb{Z}} H_2^k(\T^d)$ is a countably Hilbert nuclear space \cite[Section 1.3]{Ito1984foundations}. Its topology is defined by the family of Hilbert norms associated with the Sobolev spaces $H_2^k(\T^d)$ with $k\in\Z$.
Actually, we can also include the Hilbert norm of $H_2^{-v}(\T^d)$ into the family, resulting in the same nuclear topology on $\S(\T^d)$.
Hence, we are in the context of Theorem $1.2.4$\footnote{
Ito showed that \eqref{sigma_fields} holds not only for $\S'(\T^d)$ but for every space $S'$ that is dual to a multi-Hilbert space $S$. Note that the cylindrical $\sigma$-field is called the Kolmogorov $\sigma$-field in \cite{Ito1984foundations}, so  it is denoted by $\mathcal{B}_K(S')$.} of \cite{Ito1984foundations}, and we know that
\begin{equation} \label{sigma_fields}
	\mathcal{B}_c(H_2^{v}(\T^d)) 
	= 
	\setb{ B \cap H_2^{v}(\T^d) }{ B\in \mathcal{B}_c(\S'(\T^d)) }.
\end{equation}
Coupled with the measurability of $H_2^{v}(\T^d)$, we deduce that
\begin{equation} \label{SoboS'}
	\mathcal{B}_c(H_2^{v}(\T^d))  
	\subset 
	\mathcal{B}_c(\S'(\T^d)) .
\end{equation}

\noindent
\textbf{Step 2: $B_{p,q}^\tau(\T^d)$ is measurable.}

Let us fix $0<p,q<\infty$ and $\tau \in \R$. 
Corollary \ref{theo:isomorphism} ensures that, for $\epsilon >0$, $u\in \N$
such that $u >\max(\tau, d(1/p-1/2)_+-\tau+\epsilon)$ and $\{\Psi_{G,m}^{j,\mathrm{per}}\}$  an orthonormal basis in $L_2(\mathbb{T}^d)$  with regularity $u$, 
there exists $v \in \R$ ($v =\tau-d(1/p-1/2)_+-\epsilon $) such that
\begin{align}
	B_{p,q}^{\tau}(\T^d) 
	&= 
	\setb{ f\in H_2^{v} (\T^d) }{  
	\sum_{j \in \N, G\in G^j} 
	\parenth{ 
	\sum_{m\in \mathbb{P}_j^d} 
	\abs{  \innprod{f}{\varphi_{(j,G),m}}  }^p }^{q/p} 	< \infty} \nonumber \\
	\varphi_{(j,G),m}
	&= 
	2^{j (\alpha - d/p + d/2) + Ld/2} \psi^{j,\mathrm{per}}_{G,m} .
\end{align}
Again, we satisfy the conditions of Lemma \ref{lemma:measurable},
with $E = H_2^{v}(\T^d)$, $S = \{(j,G), \ j\in \N, G\in G^j \}$, $T_n  = \mathbb{P}_j^d$ for $n=(j,G)$, $\alpha = p$, and $\beta = q/p$.

Thus, we know that $B_{p,q}^\tau (\T^d) \in \mathcal{B}_c(H_2^{v}(\T^d))$. Considering this together with \eqref{SoboS'}, we finally deduce \eqref{eq:besov_measurable}. 

Suitable modifications are made for the cases $p=\infty$ or $q = \infty$. 
\hfill \qed \\

Let $s$ be a generalized periodic process with probability law $\mathscr{P}_s$ that satisfies $\mathscr{P}_s(B^\tau_{p,q}(\T^d)) = 1$. We define the set 
$\Omega_0 := \setb{\omega \in \Omega}{ s(\omega) \in B^\tau_{p,q}(\T^d) }$,
 and we say that 
 $s$ admits a version localized in the space $B^\tau_{p,q}(\T^d)$. It is defined by $	\tilde{s} (\omega) = s(\omega)$,  if $\omega \in \Omega_0$, and $\tilde{s} (\omega) = 0$ otherwise.
We then identify $s$ with $\tilde{s}$. As such, it becomes possible to consider random variables of the form $\langle s, f \rangle$ with $f \in (B^\tau_{p,q}(\T^d))'$ (see Theorem \ref{theo:dual}), and possibly $f \notin \S(\T^d)$. This fact will be used in the next section.

%%%%%%%%%%%%%%%%%%%%%%%%%%%%%%%%%%%%%%%%%%%%%%%%%%%%%%%%%%%%%%%%
\section{On the Regularity of L\'evy Noises} \label{noise_regularity}
%%%%%%%%%%%%%%%%%%%%%%%%%%%%%%%%%%%%%%%%%%%%%%%%%%%%%%%%%%%%%%%%

\subsection{Regularity of General L\'evy Noises}

As mentioned previously, we can only consider \emph{a priori} test functions $\varphi \in \S(\T^d)$ as
windows for a given periodic generalized  process $s$. 
However, in order to use Corollary \ref{theo:isomorphism} and measure the Besov regularity of $s$, we first derive preliminary results on Sobolev regularity. 
Here, we restrict the domain of definition of every periodic white noise so that we can consider test functions (wavelets) in $H_2^{d/2+\epsilon}(\T^d)$, and not necessarily in $\S(\T^d)$.

\begin{proposition} \label{lem:smooth}
If $w$ is a periodic white noise, and $\epsilon >0$, then 
\begin{equation}
	\mathscr{P} \parenth{w\in H_2^{-d/2-\epsilon}(\T^d)} 
	= 
	1.
\end{equation}
\end{proposition}

\begin{proof} 
According to  
Theorem A.2\footnote{
Theorem A.2 of \cite{Hida2004} ensures that, if $E$ is a countably Hilbert space included in $L_2(\R^d)$, then we can deduce bounds on the support of a probability measure over $E'$ from the continuity of the characteristic functional. Here, we use a version of this result where $E = \S(\T^d) \subset L_2(\T^d)$, the adaptation to the periodic case being straightforward. 
} of
\cite{Hida2004},
it is sufficient to prove the two following results:
\begin{enumerate}
	\item the inclusion $I : H_2^{d/2+\epsilon}(\T^d) \rightarrow L_2(\T^d)$ is a Hilbert-Schmidt operator, and
	\item the characteristic functional $\CF_w$ is continuous with respect to the norm $\lVert\cdot \rVert_{L_2(\T^d)}$.
\end{enumerate}

For the first point, we  simply remark that the family 
\begin{equation}
\set{ (1+|\bm{n}|)^{-d/2-\epsilon} \ue^{2\pi \ui \langle \bm{n}, \cdot \rangle}}_{\bm{n}\in \Z^d}
\end{equation}
is an orthonormal basis of $H_2^{d/2+\epsilon} (\T^d).$ Moreover, we have
\begin{equation}
	\sum_{\bm{n}\in \Z^d} \lVert (1+|\bm{n}|)^{-d/2-\epsilon} \ue^{2\pi \ui \langle \bm{n}, \cdot \rangle} \rVert^2_{L_2(\T^d)}
	= 
	\sum_{\bm{n}\in \Z^d} (1+|\bm{n}|)^{-d-2\epsilon}
	<
	\infty.
\end{equation}
Hence, $I$ is Hilbert-Schmidt.

Let us now show that $\CF_w$ is continuous with respect to $\lVert \cdot
\rVert_{L_2(\T^d)}$. Since $\CF_w$ is positive definite, it is sufficient to demonstrate continuity at the origin, see for instance \cite{horn1975}.  The characteristic functional of $w$ has the general form
of \eqref{eq:CF} with $f$ a L\'evy exponent, so it has a L\'evy-Khintchine representation \eqref{eq:LK}.
Let   $\set{\varphi_n}_{n\in \N}$ be  a sequence of functions in
$\S(\T^d)$ with $\lVert  \varphi_n \rVert_{L_2(\T^d)} \rightarrow 0$. If we
develop the L\'evy expansion of $f$ in \eqref{eq:CF}, then
\begin{align} \label{control_CF}
	\abs{\log \CF_w(\varphi_n)} 
	& \leq
	\abs{\mu} \norm{ \varphi_n}_{L_1(\T^d)} 
	+ \frac{\sigma^2}{2}  \norm{ \varphi_n}_{L_2(\T^d)}   \nonumber \\
	& \quad + 
	\int_{\T^d} 
	\int_{0 < \abs{t}\leq 1} 
	\abs{ \ue^{\ui \varphi_n(\rbf)t} -1 - \ui \varphi_n(\rbf) t } 
	V(\drm t)
	\lambda(\drm \rbf)   \nonumber \\
	&  \quad + 
	\int_{\T^d} 
	\int_{\abs{t}> 1} 
	\abs{\ue^{\ui \varphi_n(\rbf)t} -1}
	V(\drm t)
	\lambda(\drm \rbf) 		
\end{align}
where $\lambda$ is the Lebesgue measure on $\T^d$, normalized such that $\lambda(\T^d) = 1$.
Let $\epsilon >0$, and $M\geq1$ such that $\int_{\abs{t}> M } V(\drm t) \leq \epsilon$.
Let $n$ be large enough so that $\lVert \varphi_n \rVert_{L_2(\T^d)} \leq \epsilon/M $. 
Then, $\norm{ \varphi_n}_{L_1(\T^d)} \leq \sqrt{\lambda(\T^d)} \norm{ \varphi_n }_{L_2(\T^d)} =  \norm{ \varphi_n }_{L_2(\T^d)} \leq \epsilon / M$,

We have to control the different terms of \eqref{control_CF}. 

First, we remark that, since $M \geq 1$,
\begin{equation}
\abs{\mu} \norm{ \varphi_n}_{L_1(\T^d)} 
	+ \frac{\sigma^2}{2}  \norm{ \varphi_n}_{L_2(\T^d)} \leq (\mu + \sigma^2 /2) \epsilon.
\end{equation}

For the penultimate term, using the fact that 
$ \abs{ \ue^{\ui x} - 1 - \ui x }\leq x^2$, we have
\begin{align}
	\int_{\T^d} 
	\int_{0 < \abs{t}\leq 1} 
	\abs{\ue^{\ui \varphi_n(\rbf)t} -1 - \ui \varphi_n(\rbf) t } 
	V(\drm t)
	\lambda(\drm \rbf) 
	&\leq   
	\norm{ \varphi_n }_{L_2(\T^d)}
	\int_{0 < \abs{t}\leq 1}
	t^2 
	V(\drm t)  \nonumber \\
	& \leq 	\left( \int_{0 < \abs{t}\leq 1}
	t^2 
	V(\drm t) \right) \epsilon
\end{align}
For the last term of \eqref{control_CF}, we use $\abs{\ue^{\ui x} -1} \leq \abs{x}$
and $\abs{\ue^{\ui x} -1} \leq 2$, which yields
\begin{align}
	\int_{\T^d}
	\int_{\abs{t}> 1}
	\abs{ \ue^{\ui \varphi_n(\rbf)t} -1}
	V(\drm t)
	\lambda(\drm \rbf)
	& = 
	\int_{\T^d}
	\int_{\abs{t}> M}
	\abs{\ue^{\ui\varphi_n(\rbf)t} -1}
	V(\drm t)
	\lambda(\drm \rbf)  \nonumber\\
	& \quad +
	\int_{\T^d}
	\int_{1<\abs{t}\leq  M}
	\abs{\ue^{\ui \varphi_n(\rbf)t} -1}
	V(\drm t)
	\lambda(\drm \rbf)  \nonumber \\
	& \leq 
	2\parenth{ \int_{\T^d} \lambda(\drm \rbf) }
	\parenth{
	\int_{\abs{t}> M}
	V(\drm t) } \nonumber \\
	& \quad +
	\parenth{
	\int_{1< \abs{t} \leq M}
	\abs{t}
	V(\drm t) }
	\norm{ \varphi_n }_{L_1(\T^d)}  \nonumber \\
	& \leq 
	2\epsilon + M
	\parenth{
	\int_{1< \abs{t} \leq M}
	V(\drm t) }
	\frac{\epsilon}{M} \\
	& \leq 
	\left( 2 + \int_{1 < \abs{t}} V(\drm t) \right) \epsilon.
\end{align}
Finally, for a given $\epsilon$ and $n$ large enough, we have 
\begin{equation}
	\abs{\log \CF_w(\varphi_n)} 
	\leq 
	\kappa \epsilon,
\end{equation}
with 
\begin{equation}
	\kappa 
	= 
	\abs{\mu} 
	+ \frac{\sigma^2}{2} 
	+ \int_{\Rstar} 
	\min(1,t^2) 
	V(\drm t) 
	+ 2 
	< 
	\infty.
\end{equation} 
 Thus, we can conclude that $\CF_w$ is continuous with respect to $\norm{\cdot}_{L_2(\T^d)}$.
\end{proof}

\begin{corollary} \label{general_besov}
Let $w$ be a L\'evy white noise. Then, for every $0<p,q \leq \infty$ and $\tau\in \R$ such that
\begin{equation}  \label{eq:General1}
	\tau
	<
	d\parenth{  \frac{1}{\max\parenth{p,2}} -1  },
\end{equation}
we have
\begin{equation} \label{eq:General2}
	w \in B_{p,q}^{\tau}(\mathbb{T}^d)  \text{ a.s.}
\end{equation}
\end{corollary}

\begin{proof}
From the fourth relation of  \eqref{eq:embeddings}, we know that for all $0<p,q\leq \infty$ and $\tau>0$, we have the inclusion 
$H_2^{-d/2-\epsilon}(\T^d) \subset B_{p,q}^{-d/2-\epsilon - d(1/2-1/p)_+ - \tau}(\T^d)$. 
We already know from Proposition \ref{lem:smooth} that $w$ is a.s. in $H_2^{-d/2-\epsilon}(\T^d)$ for any $\epsilon >0$, from which we deduce \eqref{eq:General2}.
\end{proof}

Corollary \ref{general_besov} gives a general result on the Besov localization of L\'evy white noises. Moreover, we shall use Proposition \ref{lem:smooth} to obtain stronger results on the regularity of stable noises in Section \ref{subsec:stable}.

\subsection{Regularity of Compound Poisson Noises}

Compound Poisson noises are a special case of L\'evy noises. 
For all $\varphi \in \S(\T^d)$, the characteristic functional of a compound Poisson noise $w$ has the  form
\begin{equation}
	\CF_w(\varphi) = \exp \left( c  \int_{\T^d} \int_{\Rstar} (\ue^{\ui \xi t \varphi(\bm{r}) } - 1)  P(\mathrm{d}t)  \lambda(\mathrm{d}\bm{r}) \right), 
\end{equation}
where $c >0$ and $P$ is a probability measure on $\Rstar$
\cite{TaftiPoisson}.
Note that in the referenced work, compound Poisson noises are defined over $\R^d$. 
However, this definition is applicable to $\T^d$ by restriction and periodization.

A compound Poisson noise on the $d$-dimensional torus is a.s. a finite sum of Dirac delta functions with random locations and sizes. This makes the question of its Besov regularity especially simple.
Indeed, from Theorem 1 of \cite{TaftiPoisson}, we have the following equality in law:
\begin{equation}
w(\bm{r}) = \sum_{n=1}^{N} a_n \delta (\bm{r} - \bm{r}_n).
\end{equation}
In this formula, $N$ is a Poisson random variable with parameter $c$, the vector
$(a_n)\in \R^N$ is i.i.d. with law $P$, and the vector of random Dirac locations $(\bm{r}_n)\in(\R^d)^N $ satisfies that, for every  measurable $A\subset \T^d$,  
$\#\setb{n}{\bm{r}_n \in A}$ is a Poisson random variable of parameter $c\lambda(A)$. 
Note, moreover, that $(\bm{r}_n)$ and $(a_n)$ are independent. 

\begin{proposition} \label{prop:Poisson}
Let $w$ be a compound Poisson noise. Then, 
for every $\tau\in \R$ and $0<p,q \leq \infty$ such that
$ \tau < d \left( 1/p - 1 \right) $,
we have $w \in B_{p,q}^{\tau}(\mathbb{T}^d)$ almost surely.
\end{proposition}
\begin{proof}
The Besov smoothness of a single Dirac was stated in Proposition \ref{proposition_dirac}. 
Since the process $w$ is almost surely a finite sum of Diracs, it has the same regularity.  
\end{proof}

\subsection{Regularity of S$\alpha$S Noises} \label{subsec:stable}

For a given $\alpha\in (0,2]$, a \emph{symmetric-$\alpha$-stable} (S$\alpha$S)  white noise $w_{\alpha}$ is a generalized stochastic process with characteristic functional
\begin{equation}
	\CF_{w_\alpha} (\varphi) = \exp\left( -\gamma^\alpha \norm{\varphi}_{L_{\alpha}(\T^d)}^\alpha\right),
\end{equation}	
where $\gamma >0$ is the shape parameter. For $\alpha = 2$, $w_2$ is the Gaussian  white  noise.
Details on stable laws can be found in \cite{Taqqu1994}. 

\begin{theorem} \label{theo:regu_stable}
Given $\alpha \in (0,2]$, let $w_\alpha$ be a S$\alpha$S white noise. 
\begin{itemize}
\item If $\alpha = 2$, then for every $0<p,q \leq \infty$ and every $\tau < -d/2$, we have 
\begin{equation} \label{eq:Gaussian}
	w_{2} \in B_{p,q}^{\tau}(\mathbb{T}^d) \quad  a.s. ;
\end{equation}
\item if $\alpha <2$, then for every $0<p,q \leq \infty$ and $\tau\in \R$ such that
\begin{equation}  \label{eq:stable}
	\tau
	<
	d\parenth{  \frac{1}{\max\parenth{p,\alpha}} -1  },
\end{equation}
we  have
\begin{equation}\label{eq:stable2}
	w_{\alpha} \in B_{p,q}^{\tau}(\mathbb{T}^d) \quad  a.s.
\end{equation}
\end{itemize}
\end{theorem}

We first prove a result on the moments of S$\alpha$S white noises.
\begin{definition}
For every $\alpha \in (0,2]$ and $p\in (0,\infty)$, let
\begin{equation}
	C_{p,\alpha}(\varphi) 
	:= 
	\frac{\expect{  \abs{ \innprod{  w_\alpha }{\varphi } }^p}}
	{\norm{\varphi }_{L_{\alpha}(\T^d)}^p}
	\in [0,\infty]
\end{equation}
for $\varphi \in \S(\T^d)\backslash \{0\}$.
\end{definition}

\begin{lemma} \label{lem:moments}
Let $\alpha \in (0,2]$ and $p\in (0,\infty)$. The functional $C_{p,\alpha}$ is independent of $\varphi \in \S(\T^d)$, and moreover, $C_{p,\alpha}<\infty$  iff $\alpha = 2$ or $p< \alpha < 2$. 
\end{lemma}

\begin{proof}
The characteristic function of the  random variable $\innprod{ w}{\varphi} / \norm{\varphi}_{L_{\alpha}(\T^d)}$ is 
\begin{align}
	\CF_w \parenth{ \xi \frac{\varphi}{\norm{\varphi}_{L_{\alpha}(\T^d)}} } 
	&= 
	\exp \parenth{ -\gamma^\alpha 
	\norm{ \frac{\xi \varphi}{\norm{\varphi}_{L_{\alpha}(\T^d)}} }_{L_{\alpha}(\T^d)}^\alpha }  
	= \exp\left(- \gamma^\alpha |\xi|^\alpha \right).
\end{align}
Thus, the random variables $\norm{\varphi}_{L_{\alpha}(\T^d)}^{-1} \innprod{ w}{\varphi}$ are identically distributed for every $\varphi\neq 0$. From this, we deduce that
\begin{align}
	C_{p,\alpha}(\varphi)
	&=
	\frac{\expect{  \abs{ \innprod{  w }{\varphi } }^p}}
	{\norm{\varphi }_{L_{\alpha}(\T^d)}^p}  
	=
	\expect{ \abs{\norm{\varphi}_{L_{\alpha}(\T^d)}^{-1} \innprod{ w}{\varphi}}^p }
\end{align} 
is independent of $\varphi$.  Moreover, $C_{p,\alpha}$ is the $p$th-moment of a S$\alpha$S random variable. As such, it is always finite for $\alpha = 2$ (Gaussian case) and it is finite for $\alpha<2$ iff $p<\alpha$.
\end{proof} 

\noindent
\textit{Proof of Theorem \ref{theo:regu_stable}.}
We first focus on Besov spaces of the form $B^{\tau}_{p,p}(\T^d)$ with $0<p<\infty$. 
Let $u \in \N$ such that $u > \max(d/2+\epsilon,\tau,d\left( 1/p-1/2\right)_+ - \tau+\epsilon)$ for some $\epsilon>0$, and 
let $\{\Psi_{G,m}^{j,\mathrm{per}}\}$ be an orthonormal basis of $L_2(\mathbb{T}^d)$  with regularity $u$. 
From Corollary \ref{theo:isomorphism_2} and Proposition \ref{lem:smooth}, we know that $w \in B_{p,p}^\tau(\T^d)$ iff 
\begin{equation}
	h_{p,\alpha,\psi} 
	= 
	\sum_{j\geq0} 
	2^{j \parenth{\tau p-d + dp/2}} 
	\sum_{G,m} 
	\abs{ \innprod{ w}{ \psi_{G,m}^{j,per}} }^p 
	< \infty.
\end{equation}
We deduce that if 
\begin{equation}
	\expect{h_{p,\alpha,\psi}}
	=  
	\sum_{j\geq0} 
	2^{j \parenth{\tau p-d + dp/2}} 
	\sum_{G,m}
	\expect{ \abs{ \innprod{ w}{ \psi_{G,m}^{j,per}} }^p} 
	\end{equation}
is finite, then $w_{\alpha} \in B_{p,p}^{\tau}(\T)$ a.s. (of course, a positive random variable $X$ -- here the Besov norm of $w$ --  with a finite mean is a.s. finite). 

Based on Lemma \ref{lem:moments}, we know that 
\begin{equation}
	\expect{ \abs{ \innprod{ w}{ \psi_{G,m}^{j,per}} }^p}  
	= 
	C_{p,\alpha} \norm{ \psi_{G,m}^{j,per} }_{L_{\alpha}(\T^d)}^p.
\end{equation} 
Moreover, a change of variables shows that, for every $(j,G,m)$, 
\begin{equation}
	\norm{\psi_{G,m}^{j,per}}_{L_{\alpha}(\T^d)}^p 
	= 
	2^{jdp \parenth{1/2 - 1/\alpha}} \norm{\psi^{per}_{G,0}}_{L_{\alpha}(\T^d)}^p.
\end{equation}  
Then, we have
\begin{align*}
	\expect{h_{p,\alpha,\psi}} 
	&=
	\sum_{j\geq0} 
	2^{j \parenth{\tau p-d+ \frac{p}{2}}} 
	\sum_{G,m}
	\expect{ \abs{ \innprod{ w}{ \psi_{G,m}^{j,per}} }^p} \\
	&=
	\sum_{j\geq 0}
	2^{j \parenth{ \tau p - d + \frac{dp}{2} } } 
	\abs{\mathbb{P}_j^d}
	\sum_{G}
	C_{p,\alpha}
	2^{j d p \parenth{\frac{1}{2} - \frac{1}{\alpha}}}
	\norm{\psi^{per}_{G,0}}_{L_{\alpha}(\T^d)}^p  \\
	&=
	2^{Ld} C_{p,\alpha} 
	\sum_G \norm{ \psi^{per}_{G,0}}_{L_{\alpha}(\T^d)}^p 
	\parenth{
	\sum_{j\geq0}
	2^{j \parenth{\tau p + dp - \frac{dp}{\alpha} }}
	},
\end{align*}
recalling that $\mathbb{P}_j^d$ are the translation sets satisfying $\abs{\mathbb{P}_j^d}=2^{(j+L)d}$.
Finally, $\expect{h_{p,\alpha,\psi}}$ is finite iff $C_{p,\alpha}<\infty$ and $\tau p + dp - dp/\alpha = p(\tau +d-d/\alpha) < 0$, \emph{i.e.}, $\tau <d/\alpha - d$.

For $\alpha = 2$ (Gaussian case), we have $C_{p,\alpha} <\infty$ for every $p<\infty$. Hence, we obtain \eqref{eq:Gaussian} for every $p=q\neq \infty$ and $\tau <-1/2$. The other cases are deduced with the help of Proposition \ref{prop:embed}. For instance, for $\tau<-d/2$, let $0<\epsilon<- \tau - d/2$. Then, $w_2 \in B_{2,2}^{\tau+\epsilon}(\T^d)$ a.s. and $B_{2,2}^{\tau+\epsilon}(\T^d) \subset B_{\infty,\infty}^\tau(\T^d)$. 
Thus, $w_2 \in B_{\infty,\infty}^\tau(\T^d)$ a.s. 
For $0<p,q \leq \infty$ and $\tau<-d/2$, let us fix $0<\epsilon < -d/2 - \tau$. Then, $B_{p,p}^{\tau+\epsilon}(\T^d) \subset B_{p,q}^\tau(\T^d)$, and thus $w_2 \in  B_{p,q}^\tau(\T^d)$ a.s. 

Similarly, because $C_{p,\alpha}<\infty$ for $p<\alpha$, we obtain \eqref{eq:stable2} for $p=q<\alpha$ and $\tau <d/\alpha - d$. 

If  $\alpha \leq p = q \leq \infty$ and $\tau < d/p - d$, 
we define $\delta = d/p - d - \tau > 0$. 
Then we let $0<\epsilon < \delta /3$ and $p_0< \alpha$ such that $1/\alpha < 1/p_0 < 1/\alpha + \delta/3d$. 
From Proposition \ref{prop:embed}, we know that 
\begin{equation}
	B_{p_0,p_0}^{\tau +  d/p_0 - d/p + \epsilon}(\T^d) 
	\subset 
	B_{p,p}^\tau(\T^d).
\end{equation} 
Moreover, we have 
\begin{align}
	\tau+  \frac{d}{p_0}-\frac{d}{p} + \epsilon 
	& =   
	\frac{d}{p_0}-d + \epsilon - \delta \nonumber \\
	& <   
	\frac{d}{\alpha}-d.
\end{align} 
Hence, 
\begin{equation}
	w_\alpha 
	\in  
	B_{p_0,p_0}^{\tau + d/p_0 - d/p + \epsilon}(\T^d) 
	\subset 
	B_{p,p}^\tau(\T^d) \quad a.s.
\end{equation}
This proves  \eqref{eq:stable2} for $p=q$. 
The other cases follow from embedding arguments similarly to the Gaussian case. 
\hfill \qed

%%%%%%%%%%%%%%%%%%%%%%%%%%%%%%%%%%%%%%%%%%%%%%%%%%%%%%%%%%%%%%%%
\section{Discussion} \label{sec:discussion}
%%%%%%%%%%%%%%%%%%%%%%%%%%%%%%%%%%%%%%%%%%%%%%%%%%%%%%%%%%%%%%%%

The primary contributions of this paper are the smoothness estimates for non-Gaussian white noises. In Figure \ref{fig:noise_besov}, we summarize the results of the previous section, \emph{i.e.} Corollary \ref{general_besov}, Proposition \ref{prop:Poisson}, and Theorem \ref{theo:regu_stable}.

\begin{figure}[ht]  
\centering

\begin{subfigure}[b]{0.45\textwidth}
\begin{tikzpicture}[x=4cm,y=3cm,scale=0.70]

\draw[very thick, ->] (-1,0)--(0.5,0) node[circle,right] {$\frac{1}{p}$} ;
\draw[very thick, <->] (-1,-1)--(-1,0.5) node[circle,above] {$\tau$} ;

\draw[ thick,color=black]
(-1+0.05,-1/3) -- (-1-0.05,-1/3)  node[black,left] { $-\frac{d}{2}$};

\draw[ thick,color=black]
(-1+0.05,-2/3) -- (-1-0.05,-2/3)  node[black,left] { $-d$};

\draw[ thick,color=black]
(-0.25,-0.05) -- (-0.25,0.05)  node[black,above] { $1$};

\draw[ thick,color=black]
(-5/8,-0.05) -- (-5/8,0.05)  node[black,above] { $\frac{1}{2}$};
\fill[red, opacity= 0.4] (-1,-2/3) -- (-5/8,-1/3) -- (0.5,-1/3) -- (0.5,-1) -- (-1,-1) -- cycle; 
\draw[red, very thick, dashed,->](-5/8,-1/3) --(0.5,-1/3);
\draw[red, very thick, dashed,- ](-1,-2/3) -- (-5/8,-1/3);
\draw[red, very thick,->](-1,-2/3) -- (-1,-1);
\end{tikzpicture}
\caption{General L\'evy noise}
\end{subfigure}
\begin{subfigure}[b]{0.45\textwidth}
\begin{tikzpicture}[x=4cm,y=3cm,scale=0.70]

\draw[very thick, ->] (-1,0)--(0.5,0) node[circle,right] {$\frac{1}{p}$} ;
\draw[very thick, <->] (-1,-1)--(-1,0.5) node[circle,above] {$\tau$} ;

\draw[ thick,color=black]
(-1+0.05,-2/3) -- (-1-0.05,-2/3)  node[black,left] { $-d$};

\draw[ thick,color=black]
(-0.25,-0.05) -- (-0.25,0.05)  node[black,above] { $1$};

\fill[red, opacity= 0.4] (-1,-2/3) -- (0,0.25) -- (0.25,11/24) -- (0.5,11/24) -- (0.5,-1) -- (-1,-1) -- cycle; 
\draw[red, very thick, dashed,-> ](-1,-2/3) -- (0.25,11/24);
\draw[red, very thick,->](-1,-2/3) -- (-1,-1);

\end{tikzpicture}
\caption{Compound Poisson}
\end{subfigure}

\begin{subfigure}[b]{0.45\textwidth}
\begin{tikzpicture}[x=4cm,y=3cm,scale=0.70]

\draw[very thick, ->] (-1,0)--(0.5,0) node[circle,right] {$\frac{1}{p}$} ;
\draw[very thick, <->] (-1,-1)--(-1,0.5) node[circle,above] {$\tau$} ;

\draw[ thick,color=black]
(-1+0.05,-1/3) -- (-1-0.05,-1/3)  node[black,left] { $-\frac{d}{2}$};

\fill[red, opacity= 0.4] (-1,-1) -- (-1,-1/3) -- (0.5,-1/3) -- (0.5,-1) -- cycle; 
\draw[red, very thick,-> ](-1,-1/3) -- (-1,-1);
\draw[red, very thick, dashed,->](-1,-1/3) --(0.5,-1/3);

\draw[ thick,color=black]
(-0.25,-0.05) -- (-0.25,0.05)  node[black,above] { $1$};
\end{tikzpicture}
\caption{Gaussian}
\end{subfigure}
\begin{subfigure}[b]{0.45\textwidth}
\begin{tikzpicture}[x=4cm,y=3cm,scale=0.70]

\draw[very thick, ->] (-1,0)--(0.5,0) node[circle,right] {$\frac{1}{p}$} ;
\draw[very thick, <->] (-1,-1)--(-1,0.5) node[circle,above] {$\tau$} ;

\draw[ thick,color=black]
(-1+0.05,-2/3) -- (-1-0.05,-2/3)  node[black,left] { $-d$};

\draw[ thick,color=black]
(-0.25,-0.05) -- (-0.25,0.05)  node[black,above] { $1$};

\fill[red, opacity= 0.4] (-1,-2/3) -- (0,0.25) -- (0.5,0.25) -- (0.5,-1) -- (-1,-1) -- cycle; 
\draw[red, very thick, dashed,->](0,0.25) --(0.5,0.25);
\draw[red, very thick, dashed,- ](-1,-2/3) -- (0,0.25);
\draw[red, very thick,->](-1,-2/3) -- (-1,-1);

\draw[ thick,color=black]
(0,0.05) -- (0,-0.05)  node[black,below] {$\frac{1}{\alpha}$};

\end{tikzpicture}
\caption{S$\alpha$S, $\alpha <1$}
\end{subfigure}

\caption{These diagrams represent the Besov localization of different white noises. 
A white noise is in a given Besov space $B_{p,q}^\tau(\T^d)$ a.s. if $(1/p,\tau)$ is located in the shaded region of the diagram.} 
\label{fig:noise_besov}
\end{figure}
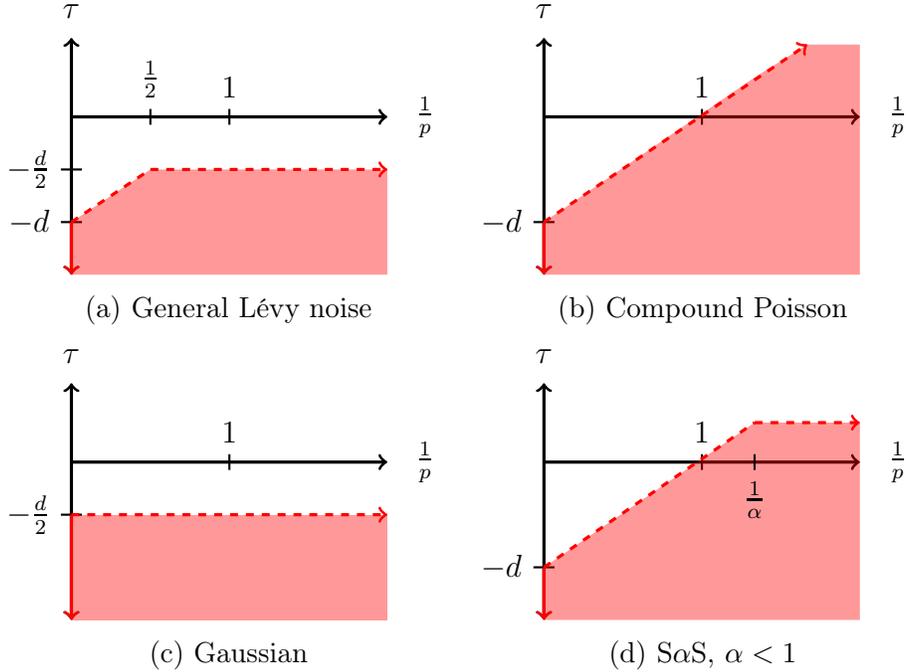

As a point of fact, the results in the Gaussian case are not new. Indeed, Veraar obtained \eqref{eq:Gaussian} in Theorem 3.4 of \cite{veraar2010regularity}, and his result is actually stronger because he showed that $w_2 \in B_{p,\infty}^{-d/2}(\T^d)$ almost surely for $1\leq p<\infty$. 
In our case, we only have results for regularity $\tau<-d/2$. Moreover, Veraar derived some converse results, showing, for instance, that $w_2$ is almost surely not in $B_{p,q}^{-d/2}(\T^d)$ if $q\neq \infty$. This slight refinement is largely due to the fact that his approach is intrinsically Gaussian in nature (the white noise is described by its independent collection of Fourier coefficients, see Proposition 3.3 of \cite{veraar2010regularity}). Unfortunately, these methods do not naturally extend to the non-Gaussian setting.   

By contrast, our approach allows us to directly work with Besov parameters $p,q< 1$ that are not considered in \cite{veraar2010regularity}.
For instance, we see that in the proof of Theorem \ref{theo:regu_stable} for $\alpha<1$, we deduce all of our regularity results from the preliminary cases with $p=q <\alpha<1$. 

Symmetric-$\alpha$-stable noises are natural generalizations of Gaussian white noises.
They are a parametric family defined by the parameter $\alpha \in (0,2]$, where $\alpha = 2$ corresponds to the Gaussian case  \cite{Taqqu1994}. 
In Theorem \ref{theo:regu_stable}, we characterized their smoothness behavior.
As $\alpha$ becomes smaller, we see in particular that the regularity of $w_\alpha$ increases.
When $\alpha$ goes to zero, the shaded region in Figure \ref{fig:noise_besov}(d) converges to the shaded region for a compound Poisson noise, Figure \ref{fig:noise_besov}(b).

While we have determined smoothness bounds, we have not yet exhausted the topic of regularity for periodic L\'evy noises.
First, we conjecture a $0-1$ law for the Besov regularity of periodic noises.
\begin{conjecture}
For every noise $w$ and every $\tau\in \R$, $0<p,q\leq \infty$, 
\begin{equation*}
\mathbb{P} \left(w \in B_{p,q}^\tau(\T^d) \right) = 0  \quad  \text{or} \quad 1.
\end{equation*}
\end{conjecture}
Second, in the case of $\alpha$-stable noises, we obtained results for regularity smaller than $1/\alpha -1$. 
However, it is shown in Theorem IV.1 of \cite{ciesielski1993quelques} that S$\alpha$S processes on the real line ($d=1$) with $1<\alpha<2$ are in the local Besov spaces $B_{p,\infty,\mathrm{loc}}^{\alpha}(\R)$ if $1\leq p< \alpha$. This suggests the following.
\begin{conjecture}
If $0<p<\alpha<2$, the S$\alpha$S processes $w_{\alpha}$  satisfy
\begin{equation*}
\mathbb{P} \left(w_{\alpha} \in B_{p,\infty}^{d(1/\alpha-1)}(\T^d) \right) =  1. 
\end{equation*}
\end{conjecture}
We expect finer results  than Corollary \ref{general_besov} for the regularity of general L\'evy noises (non-stable, non-compound Poisson). Ideally, we should be able to compute the probability of every Besov space for each L\'evy noise.  
Finally, in a future work, we plan to extend the results of this paper to sparse stochastic processes that are solutions of non-Gaussian stochastic partial differential equations.

\bibliographystyle{plain}
\bibliography{refs2}

\end{document}